\documentclass[final]{colt2024}

\title[Open Problem: Anytime GD]{Open Problem: Anytime Convergence Rate of Gradient Descent}
\usepackage{times}

\newtheorem{openproblem}{Open Problem}

\usepackage[capitalize]{cleveref}
\crefformat{none}{(#2#1#3)}
\crefname{equation}{Eq.}{Equations}
\crefname{eq}{Eq.}{Equations}
\crefname{fact}{Fact}{Facts}
\crefname{lemmma}{Lemma}{Lemmas}
\crefname{lemma}{Lemma}{Lemma}
\crefname{lem}{Lemma}{Lemma}
\crefname{openproblem}{Open Problem}{Open Problem}
\crefname{figure}{Figure}{Figures}
\crefname{defn}{Definition}{Definitions}
\crefname{ineq}{Inequality}{Inequalities}
\crefname{prob}{Problem}{Problems}
\crefname{corollary}{Corollary}{Corollaries}
\creflabelformat{ineq}{#2{\upshape(#1)}#3} 
\crefalias{thm}{theorem}
\Crefname{listfact}{Fact}{Facts}
\crefname{assum}{Assumption}{Assumptions} 

\newcommand{\reals}{\mathbb{R}}

\newcommand{\abs}[1]{\left| #1 \right|}

\usepackage{xcolor}

\newcommand{\beq}{\begin{eqnarray*}}
\newcommand{\eeq}{\end{eqnarray*}}
\newcommand{\beqn}{\begin{eqnarray}}
\newcommand{\eeqn}{\end{eqnarray}}

\newcommand{\NN}{\mathbb{N}}

\newcommand{\half}{\frac{1}{2}}

\newcommand{\norm}[1]{\left\|#1\right\|}

\coltauthor{\Name{Guy Kornowski} \Email{guy.kornowski@weizmann.ac.il}\\
 \Name{Ohad Shamir} \Email{ohad.shamir@weizmann.ac.il}\\
 \addr Weizmann Institute of Science}

\begin{document}

\maketitle

\begin{abstract}%
  Recent results show that vanilla gradient descent can be accelerated for smooth convex objectives, merely by changing the stepsize sequence. We show that this can lead to surprisingly large errors indefinitely, and therefore ask: Is there any stepsize schedule for gradient descent that accelerates the classic $\mathcal{O}(1/T)$ convergence rate, at \emph{any} stopping time $T$?
\end{abstract}


\section{Introduction}

Consider the classic setting of optimizing a smooth convex objective via gradient descent (GD):
Given a convex function $f:\reals^d\to\reals$ which is $L$-smooth (i.e. $\nabla f$ exists and is $L$-Lipschitz),
and an initial point $x_0\in\reals^d$, the GD iterates with stepsizes $(\eta_t)_{t=0}^{\infty}$ are defined as $x_{t+1}=x_{t}-\eta_t \nabla f(x_t)$.

The textbook analysis of GD under this setting (e.g. \citealp{nesterov2018lectures,bubeck2015convex}) asserts that when the stepsize schedule is fixed to be constant $\eta_t\equiv \overline{\eta}\in(0,\frac{2}{L})$, the iterates satisfy the bound
\begin{equation} \label{eq: classic GD}
f(x_T) -f^*\lesssim \frac{L\norm{x_0-x^*}^2}{T}~~~\text{for all }T\in\NN
~,
\end{equation}
where $f^*=\inf f$, $x^*$ is any minimizer of $f$,
and ``$\lesssim$'' hides a  constant. It is also well-known that for constant steps larger than $2/L$, the algorithm can diverge. 
The behavior of GD in this setting is extremely well-studied, and one would think that it is fully understood.

However, quite unexpectedly, a recent line of work established that GD can achieve faster convergence rates than implied by \cref{eq: classic GD}, without any modification to the algorithm itself, merely by using appropriate \emph{non-constant} stepsize schedules which incorporate occasional \emph{long steps}, 
larger than $2/L$ \citep{grimmer2023accelerated,altschuler2023acceleration}.\footnote{\citep{daccache2019performance,eloi2022worst,das2024branch,grimmer2023provably} previously improved the \emph{constant} factor in \eqref{eq: classic GD}.}
In particular, \citet{altschuler2023silverII} constructed a stepsize sequence, coined the ``silver stepsize'' schedule, that guarantees
\begin{equation}\label{eq: silver guarantee}
    f(x_T)- f^*
    \lesssim \frac{L\norm{x_0-x^*}^2}{T^{\log_2(1+\sqrt{2})}}
    \approx \frac{L\norm{x_0-x^*}^2}{T^{1.2716}}~~~\text{for all }T=2^n-1,n\in\NN
~.
\end{equation}
Remarkably, compared to \cref{eq: classic GD}, this bound achieves an accelerated $o(1/T)$ rate, merely by changing the GD stepsize schedule. However, note that the bound no longer applies for all $T$. Instead, it only applies for certain exponentially-increasing horizons $T=2^n-1$, with no guarantee on the performance of intermediate iterates. From a practical viewpoint, this is not quite satisfactory, as often the number of iterations is not carefully chosen in advance. Although this can be circumvented by doubling tricks or tracking the best iterate obtained so far (as further discussed below), in practice it is desirable to have a uniform, monotonically-decreasing guarantee on the error, which ensures that at any large enough stopping point, the resulting optimization error is small. As far as we know, none of the existing results for stepsize-based acceleration apply in an anytime fashion, and it is not  clear that such acceleration is even possible. Hence, we formulate the following open problem:
\begin{openproblem} \label{open problem}
    What is the best \textbf{anytime} convergence rate achievable by GD with some stepsize sequence $(\eta_t)_{t=0}^{\infty}$, uniformly over $L$-smooth convex functions? In particular, is there any stepsize sequence and some $\alpha>1$,
    such that for all $L$-smooth convex $f$, 
\begin{equation} \label{eq: acc bound}
f(x_T)-f^*\lesssim \frac{L\norm{x_0-x^*}^2}{T^{\alpha}}~~~\text{for all }T\in\NN
~?
\end{equation}
\end{openproblem}

We remark that we seek an anytime, monotonically decreasing \emph{upper bound} on the error: Indeed, with long steps, the errors themselves may not decrease monotonically. A similar phenomenon is exhibited by Nesterov's  accelerated gradient method \citep{nesterov1983method}: It is well-known (e.g., \citealp{d2021acceleration}) that this algorithm does \emph{not} monotonically decrease the error, while still having an anytime, monotonically decreasing error bound similar to \cref{eq: classic GD} (replacing $T^{-1}$ by $T^{-2}$, which is the optimal dimension-free rate for gradient-based algorithms). 

We further note that while we focus on the convex setting, the analogous question for the \emph{strongly}-convex case is also of interest, for which it is unclear whether any stepsize sequence achieves an anytime $\exp(-T/o(\kappa)))$ rate, uniformly over functions with condition number $\kappa$.

\paragraph{Equivalent view: Bounds on the iterate vs. best iterate}
A bound such as \eqref{eq: silver guarantee} can be easily converted to an anytime bound on $\min_{t\in[T]}f(x_t)-f^*$ (namely, the best iterate obtained so far): Indeed, for any given $T\in\NN$, let $\hat{T}$ be the largest integer such that $\hat{T}\leq T$ and $\hat{T}=2^n-1$ for some $n\in\NN$. It is easy to show that $\hat{T}\geq T/2$, and hence by \eqref{eq: silver guarantee},
$
\min_{t\in[T]}f(x_t)-f^*~\leq~
f(x_{\hat{T}})-f^*\lesssim\frac{L\norm{x_0-x^*}^2}{\hat{T}^{1.2716}}\leq
\frac{2^{1.2716}L\norm{x_0-x^*}^2}{T^{1.2716}}~~~\text{for all }T\in\NN~.
$
This is an anytime guarantee, which matches \eqref{eq: silver guarantee} up to a small numerical constant (indeed, \cite{grimmer2023accelerated} even present their accelerated result in this manner). However, this anytime guarantee no longer applies to individual iterates $x_T$. Thus, our question is equivalent to asking whether an appropriate stepsize schedule can accelerate GD in terms of $f(x_T)-f^*$, rather than $\underset{{t\in[T]}}{{\min}}\,f(x_t)-f^*$.

\section{Preliminary results}

We take steps towards the resolution of \cref{open problem} by providing two results, both of which hold already in dimension $d=1$. These results indicate the tension between acceleration with GD and anytime guarantees.
Moreover, we establish that current accelerating stepsize schedules can strongly fail to meet anytime guarantees.

First, we note that long steps are not only required for anytime acceleration, in fact such acceleration necessitates \emph{arbitrarily large} steps.\footnote{Note that this assertion relies on the finite-time bound being uniform over all smooth convex functions. In fact, an asymptotic $o(L/T)$ for any \emph{fixed} function actually holds for constant stepsizes \citep{lee2019first}.} This can be formally stated as follows:

\begin{theorem} \label{thm: infty}
    Suppose that GD with stepsizes $(\eta_t)_{t=0}^{\infty}$ satisfies an accelerated uniform guarantee, namely $\forall T\in\NN:f(x_T)-f^*\lesssim o(L/T)$ for any $L$-smooth convex $f$. Then ${\lim\sup}_{t\to\infty}\eta_t=\infty$.
\end{theorem}

Indeed, the step size schedules in \cite{grimmer2023accelerated} and \cite{altschuler2023silverII} satisfy this requirement: They both involve a fractal-like stepsize schedule, where the stepsize value increases exponentially at exponentially-increasing intervals. However, our next result implies that occasional huge steps (compared to previous steps) can prevent any decaying uniform bound, even one which is not accelerated. As a corollary (see Corollary~\ref{cor:silver}), this implies that the silver stepsize schedule cannot enjoy \emph{any} convergence guarantee which holds in an anytime fashion.

\begin{theorem} \label{thm: main}
    For any $L>0$, stepsize schedule $(\eta_t)_{t=0}^{\infty}$ and $T\in\NN$ satisfying $\min\{\frac{\eta_T}{2},\sum_{t=0}^{T-1}\eta_{t}\}\geq \frac{1}{L}$, there exists an $L$-smooth convex  $f:\reals\to\reals$ such that $f(x_{T+1})-f^*\geq \frac{1}{32}L\norm{x_0-x^*}^2\left(\frac{\eta_T}{\sum_{t=0}^{T-1}\eta_t}\right)^2$.
\end{theorem}

Note that the lower bound holds for $T$ as long as it satisfies $\eta_T \geq \frac{2}{L}$ and $\sum_{t=0}^{T-1}\eta_{t}\geq \frac{1}{L}$. The latter condition is in a sense generic, and should be expected to hold for all large enough $T$, since otherwise GD cannot guarantee convergence to possibly far-away minima in the first place. Thus the important condition is that $\eta_T \geq \frac{2}{L}$, namely that at time $T$ GD takes a long step (beyond the $2/L$ regime). The theorem formally shows that long steps may ``overshoot'', as measured by the squared ratio $(\frac{\eta_T}{\sum_{t=0}^{T-1}\eta_t})^2$. In particular, \emph{the larger this ratio is, the larger the error can be after the long step}.

\begin{corollary} \label{cor: not_too_large}
    If a stepsize sequence $(\eta_t)_{t=0}^{\infty}$ satisfies an anytime accelerated bound as in \cref{eq: acc bound}, then 
    $\eta_T\lesssim \frac{\sum_{t=0}^{T-1}\eta_t}{T^{\alpha/2}}
    =o\left(\frac{\sum_{t=0}^{T-1}\eta_t}{\sqrt{T}}\right)
    $ for infinitely many $T\in\NN$ in which long steps occur.
\end{corollary}

Overall, we see that acceleration requires the stepsize sequence to have a subsequence going to infinity by \cref{thm: infty}, yet ``not too fast'', as captured by Corollary~\ref{cor: not_too_large}.
Furthermore, \cref{thm: main} shows that if a stepsize schedule that incorporates long steps satisfies
$\left(\frac{\eta_T}{\sum_{t=0}^{T-1}\eta_t}\right)^2\gtrsim 1$
for infinitely many $T\in\NN$, then the lower bound
$f(x_{T+1})-f^*\gtrsim L\norm{x_0-x^*}^2
$
applies for arbitrarily large $T\in\NN$.\footnote{This is the strongest possible lower bound, since $f(x_0)-f^*\leq \frac{L}{2}\norm{x_0-x^*}^2$ in the first place due to smoothness.}
In particular, it is easy to verify that the silver stepsize satisfies this property \citep[Eq. 1.3 and Lemma 2.3]{altschuler2023silverII}, hence we get:

\begin{corollary} \label{cor:silver}
    No anytime bound of the silver stepsize schedule goes to zero (at any rate whatsoever).
\end{corollary}

\section{Proofs}

\subsection{Proof of \cref{thm: infty}}

We first note that by rescaling, it suffices to prove the claim for $L=1$. We will also assume $\sum_{t=0}^{\infty}\eta_t=\infty$ (otherwise it is well-known that GD may not converge). Next, consider the convex quadratic $f_T(x)=\frac{x^2}{2\sum_{t=0}^{T-1}\eta_t}$ which is minimized at $f_T(0)=0$,
and note that for sufficiently large $T$ we may assume without loss of generality that $f_T$ is 1-smooth, and $\forall t<T:\eta_{t}\leq \tfrac{1}{2}\sum_{j=0}^{T-1}\eta_j$.\footnote{Otherwise, since $\sum_{t=0}^{\infty}\eta_t=\infty$, there is a subsequence of stepsizes diverging to $\infty$, proving the theorem altogether.} For $x_0=1$, a simple induction reveals that
$x_T
=\prod_{t=0}^{T-1}\left(1-\frac{\eta_t}{\sum_{t=0}^{T-1}\eta_t}\right)$.
So if for all $T\in\NN:f(x_T)-f^*=f(x_T)\leq\phi(T)$ for some $\phi(T)=o(1/T)$,
then
\begin{align*}
\phi(T)
&\geq f_T(x_T)
=\frac{1}{2\sum_{t=0}^{T-1}\eta_t}\cdot\prod_{t=0}^{T-1}\left(1-\frac{\eta_t}{\sum_{t=0}^{T-1}\eta_t}\right)^2
\\&=\frac{1}{2\sum_{t=0}^{T-1}\eta_t}\cdot\exp\left[2\cdot\sum_{t=0}^{T-1}\log\left(1-\frac{\eta_t}{\sum_{t=0}^{T-1}\eta_t}\right)\right]
\\&\geq \frac{1}{2\sum_{t=0}^{T-1}\eta_t}\cdot\exp\left[-4\cdot\sum_{t=0}^{T-1}\frac{\eta_t}{\sum_{t=0}^{T-1}\eta_t}\right]
=\frac{e^{-4}}{2\sum_{t=0}^{T-1}\eta_t}~,
\end{align*}
thus
$\max_{0\leq t\leq T-1}\eta_t
\geq\frac{1}{T}\sum_{t=0}^{T-1}\eta_t
\geq \frac{e^{-4}}{2T\phi(T)}
\overset{T\to\infty}{\longrightarrow}\infty$.

\subsection{Proof of \cref{thm: main}}

We first note that by rescaling, it suffices to prove the claim for $L=1$.
Accordingly, we let $T\in\NN$ be so that
$\min\{\frac{\eta_T}{2},\sum_{t=0}^{T-1}\eta_{t}\}\geq1$, and
denote $\Sigma_T:=\sum_{t=0}^{T-1}\eta_t,~
c_T:=\frac{\eta_T^2}{32\Sigma_T^2}$.

Let $a,r>0$ to be determined later, and consider the scaled Huber loss and initialization point:
\begin{align*}
f(x)
=
\begin{cases}
    \frac{a}{2} x^2~, & x\leq r
    \\
    ar\cdot x-\frac{ar^2}{2}~, & x> r
\end{cases}
~,~~~~~
x_0=r+ar \Sigma_T~.
\end{align*}
Note that $f$ is convex, $a$-smooth, and that $f^*=f(0)=0$.
We will show that for a suitable choice of $a\leq 1,~r>0:$
\begin{equation} \label{eq: strict lower bound silver}
    f(x_{T+1})-f^*\geq
    c_T\norm{x_0-x^*}^2=
    \frac{1}{32}\norm{x_0-x^*}^2\left(\frac{\eta_T}{\sum_{t=0}^{T-1}\eta_t}\right)^2~.
\end{equation}

\begin{lemma} \label{lem:ar}
    There exist $a,r>0$ such that $\max\left\{\frac{2}{\eta_{T}},~\left(\frac{8c_T}{\eta_T^2 r^2}\right)^{1/3}\right\}\leq a\leq \min\left\{1,~\frac{1-r}{r\Sigma_T}\right\}$.
\end{lemma}

\begin{proof}
By assumption on $T$ that $\eta_T\geq 2$,
and the definition of $c_T=\frac{\eta_T^2}{32\Sigma_T^2}\iff \sqrt{8c_T}=\frac{\eta_T}{2\Sigma_T}$ we get $\frac{\sqrt{8c_T}}{\eta_T}
\leq
\min\left\{\sqrt{\eta_T c_T},~\frac{1}{2\Sigma_T}
\right\}$.
Thus, there exists some $r>0$ such that
\begin{equation} \label{eq: r 1st}
\frac{\sqrt{8c_T}}{\eta_T}
\leq r
\leq\min\left\{\sqrt{\eta_T c_T},~\frac{1}{2\Sigma_T}\right\}~.
\end{equation}
Fixing such $r$, and recalling that $\Sigma_T\geq 1$ by assumption on $T$, we get that $r\leq \frac{1}{2\Sigma_T}\leq\half$ which implies $\tfrac{1}{8}\leq(1-r)^3$,
thus $\frac{1}{2\Sigma_T}\leq \frac{4(1-r)^3}{\Sigma_T}=\frac{\eta_{T}^2(1-r)^3}{8c_T\Sigma_T^3}$. Combining this with \cref{eq: r 1st}, we get the cruder upper bound $\frac{\sqrt{8c_T}}{\eta_T}
\leq r\leq
\min\left\{\sqrt{\eta_{T}c_T},~\frac{\eta_{T}^2(1-r)^3}{8c_T\Sigma_T^3}\right\}$.
Rearranging the latter inequalities, we get that
$\frac{2}{\eta_{T}}
\leq\left(\frac{8c_T}{\eta_{T}^2r^2}\right)^{1/3}
\leq\min\left\{1,~\frac{1-r}{r\Sigma_T}\right\}$, so in particular
$\max\left\{\frac{2}{\eta_{T}},~\left(\frac{8c_T}{\eta_{T}^2r^2}\right)^{1/3}\right\}
\leq \min\left\{1,~\frac{1-r}{r\Sigma_T}\right\}$.
Thus, setting $a$ between these left hand side and right hand side completes the proof.
\end{proof}

Following Lemma~\ref{lem:ar}, we consider $a,r$ that satisfy the conditions stated therein.
That being the case, since $a\leq 1$ we see that $f$ is indeed $1$-smooth. Furthermore, we have for all $t\leq T:x_{t+1}=x_t-\eta_t f'(x_t)=x_{t}-ar\eta_t$,
thus $x_{T}=x_{0}-ar \sum_{j=0}^{T-1}\eta_j=r$.
This implies, by the gradient descent update and the definition of $f$, that
$x_{T+1}=x_{T}-\eta_{T}f'(x_{T})
=r-ar\eta_{T}=r(1-a\eta_T)$.
Further noting that by Lemma~\ref{lem:ar} it holds that $\frac{2}{\eta_T}<a$ which implies $a\eta_T-1\geq1$, we overall get that
\begin{align*}
f(x_{T+1})-f^*=f(x_{T+1})
&=\frac{a}{2}(x_{T+1})^2
=\frac{a(a\eta_T-1)^2 r^2}{2}
\\
&\overset{(\star)}{\geq}\frac{a^3 \eta_T^2 r^2}{8}
\overset{(\star\star)}{\geq}c_T
\overset{(\star\star\star)}{\geq}c_T\abs{r+ar \Sigma_T}^2
=c_T\abs{x_0-x^*}^2~,
\end{align*}
where $(\star),(\star\star),(\star\star\star)$ all follow from Lemma~\ref{lem:ar}, thus establishing \cref{eq: strict lower bound silver}, completing the proof.

\acks{This research is supported in part by European Research Council (ERC) grant 754705. GK is supported by an Azrieli Foundation graduate fellowship.}

\bibliography{bib}

\end{document}